\documentclass{amsart}
\usepackage{graphicx}
\usepackage{latexsym}
\usepackage{color}
\usepackage{amsfonts}
\usepackage[all]{xy}
\usepackage{amssymb, mathrsfs, amsfonts, amsmath}
\usepackage{amsbsy}
\usepackage{xcolor}
\usepackage{amsmath
	,amsthm
	,amsfonts
	,amssymb
	,amscd
} 
\usepackage[mathscr]{eucal}
\usepackage{multicol}
\usepackage[bottom]{footmisc} 
\usepackage{dsfont}
\usepackage{float} 
	\usepackage{graphicx} 

\usepackage{amsfonts}
\newtheorem{definition}{Definition}

\newtheorem{claim}{Claim}

\newtheorem{remark}{Remark}
\newtheorem{theorem}{Theorem}
\newtheorem{example}{Example}
\numberwithin{equation}{section}

\newcommand{\be}{\begin{equation}}
	\newcommand{\ee}{\end{equation}}
\newcommand{\ben}{\begin{enumerate}}
	\newcommand{\een}{\end{enumerate}}
\newcommand{\beq}{\begin{eqnarray}}
	\newcommand{\eeq}{\end{eqnarray}}
\newcommand{\beqn}{\begin{eqnarray*}}
	\newcommand{\eeqn}{\end{eqnarray*}}

\begin{document}
	\title{Geometry of a Navigation problem: The $\lambda-$Funk Finsler Metrics}
	\author{Newton Sol\'orzano, V\'ictor Le\'on, Alexandre Henrique and Marcelo Souza}
	
\address{N. Sol\'orzano. ILACVN - CICN, Universidade Federal da Integração Latino-Americana, Parque tecnológico de Itaipu, Foz do Iguaçu-PR, 85867-970 - Brazil}
	\email{nmayer159@gmail.com}

\address{V. Le\'on. ILACVN - CICN, Universidade Federal da Integração Latino-Americana, Parque tecnológico de Itaipu, Foz do Iguaçu-PR, 85867-970 - Brazil}
\email{victor.leon@unila.edu.br}

\address{A. Henrique. ILACVN - CICN, Universidade Federal da Integração Latino-Americana, Parque tecnológico de Itaipu, Foz do Iguaçu-PR, 85867-970 - Brazil}
\email{alexandrehrfilho@gmail.com}

\address{M. Souza. Universidade Federal de Goiás - UFG, GO, Brazil}
\email{msouza\underline{ }2000@yahoo.com}

	

	\begin{abstract}
We investigate the travel time in a navigation problem from a geometric perspective. The setting involves an open subset of the Euclidean plane, representing a lake perturbed by a symmetric wind flow proportional to the distance from the origin. The Randers metric derived from this physical problem generalizes the well-known Euclidean metric on the Cartesian plane and the Funk metric on the unit disk. We obtain formulas for distances, or travel times, from point to point, from point to line, and vice-versa.	\end{abstract}
	
	\date{\today}

	\maketitle
\section{Introduction}
The Randers Metrics are important Finsler metrics, defined as the sum of a Riemann metric and a $1-$form. These metrics were first studied by the physicist G. Randers in 1941 from the standard point of general relativity \cite{Randers1941}. Later on, these metrics were applied to the theory of the electron microscope by R. S. Ingarden in 1957, who first named them Randers metrics. Since then, Randers metrics have been used in many areas like Biology, Ecology, Physics, Seismic Ray Theory, etc.

The Zermelo Navigation problem came to Zermelo's mind when the airship ``Graf Zeppelin'' circumnavigated the earth in August 1929. He considered a vector field given in the Euclidean plane that describes the distribution of winds as depending on place and time and treated the question of how an airship or plane, moving at a constant speed against the surrounding air, has to fly in order to reach a given point $Q$ from a given point $P$ in the shortest time possible \cite{Heinz2015}.

In \cite{DBao2004} the authors described Zermelo's navigation problem on Riemannian manifolds and showed that the path with shortest travel time is the geodesic of Randers metrics. Conversely, they showed constructively that every Randers metric arises as a solution to Zermelo's navigational problem on some Riemannian landscape under the influence of an appropriate wind. See Proposition 1.1 in Section 1.3 in \cite{DBao2004}. The Funk metric on the unit n-dimensional ball $\mathbb{B}^n(1)$, which is one of the most important Randers metrics, can be obtained by perturbing the Euclidean metric $||.||$ by the vector field $W_x=-x$.   Particularly, in \cite{Chavez2021}, the authors considered a lake in the shape of the unit disk $\mathbb{B}^2$, with the concentric and symmetric wind current given by the vector field $W(x_1,x_2)=(-x_1,-x_2)$. The distance function (or travel time) in this context, for $P\neq Q$, is given by:
\begin{equation}\label{eq1}
	d_F(P,Q)=\ln\left(\frac{\sqrt{\langle P, Q-P \rangle^2+(1-\Vert P\Vert^2)\Vert Q-P\Vert^2}-\langle P, Q-P \rangle}{\sqrt{\langle P, Q-P\rangle^2+(1-\Vert P\Vert^2)\Vert Q-P\Vert^2}-\langle Q, Q-P \rangle}
	\right), 
\end{equation}
where $\langle\cdot,\cdot\rangle$ and $\Vert \cdot \Vert$ are the usual inner product and the usual Euclidean norm, respectively, and $d_F(P,P)=0$. Also, in \cite{Chavez2021}, equations for the circle and formulas for the distance from a point to a line and from a line to a point were also obtained. Other geometric properties were studied in \cite{Shen2001} for more general Funk metrics.

In this work, we consider the wind current given by the vector field $W_\lambda(x_1,x_2)=\lambda(-x_1,-x_2),$ where $\lambda \geq 0.$ (For $\lambda<0$ the study is analogous). With this, we define the $\lambda-$Funk metric. Naturally, when $\lambda=0,$ we obtain the Euclidean norm; when $\lambda=1$, we get the Funk metric on $\mathbb{B}^2$. We note the $\lambda-$Funk metric is spherically symmetric Finsler metric. In Section 2 we recall some basic results for the well development of the work. In Section 3 we obtain and define the $\lambda-$Funk metric,  recalling results about spherically symmetric Finsler metrics we prove that their geodesics are straight lines. In Section 4 we obtain the $\lambda-$Funk distance (or traveling time), and we give some properties such as their non-symmetry. In Section 5 we classify the circumference, and with this, we obtain formulas for the distance from point to line and from line to point. 
\section{Preliminaries}
In this section, some definitions and results necessary for the development of our work are introduced. We adopt the definitions given in \cite{Chavez2021}, which can also be found in \cite{Carmo2019}, such as inner product, norm, regular curve, arc length (which will be referred to as usual or Euclidean), and vector field.

In the following, $\mathbb{R}^2$ denote the Cartesian plane, which is the set of ordered pairs $(x_1,x_2),$ where $x_1, x_2 \in \mathbb{R}.$

We present a simplified definition of a Finsler metric on an open subset $U$ of $\mathbb{R}^n$. For a more comprehensive treatment of Finsler metrics on general manifolds, we refer to \cite{GuoMo2018,ChengShen2012,Shen2001}.

\begin{definition}\label{def:Finsler}
	The function $F:U\times \mathbb{R}^n\to \mathbb{R},$ where $U\subset \mathbb{R}^n$, is called Finsler metric on $U$, if for $x\in U$ and $y\in \mathbb{R}^n$, $F$ satisfies the following properties:
	\begin{enumerate}
		\item $F(x,y)$ is $C^{\infty}$ for all $x\in U$ and $y\neq 0$;
		\item $F(x,y)>0$, for all $x\in U$ and $y\neq 0$;
		\item $F(x,\delta y) = \delta F(x,y)$, where $\delta$ is any positive real number;
		\item The Hessian matrix of $\frac{1}{2}F^2$, denoted by $[g_{ij}]$,
		\begin{align}
			\left[g_{ij}\right]&=\left[\frac{1}{2}\frac{\partial^2F^2}{\partial y_i \partial y_j}\right]\nonumber
		\end{align} is positive definite. 
	\end{enumerate}
\end{definition}


For example, Euclidean norm, Riemannian metrics or the Funk metric defined below are Finsler metrics.
\begin{definition}\label{FunkMetric}{\em 
		Let $ x=(x_1,x_2)\in{\Omega}\subset \mathbb{R}^2$ and $y=(y_1,y_2)\in\mathbb{R}^2$. 
		The function
		\begin{align*}
			F(x,y)=&\sqrt{a_{11}(x)y_1^2+ 2a_{12}(x)y_1y_2 + a_{22}(x)y_2^2}+b_1(x)y_1+b_2(x)y_2,
		\end{align*}
		where $b_1= \dfrac{x_1}{1-x_1^2-x_2^2}$, $b_2= \dfrac{x_2}{1-x_1^2-x_2^2}$
		and
		\begin{displaymath}
			[a_{ij}]=\frac{1}{(1-x_1^2-x_2^2)^2}\left(\begin{array}{cc}
				{1-x_2^2}  & {x_1x_2} \\
				{x_1x_2} & {1-x_1^2}
			\end{array}\right),
		\end{displaymath}
		is called the \textit{Funk metric on the unit disk} $ \mathbb{B}^2=\{x\in\mathbb{R}^2;\; x_1^2+x_2^2<1\}$.}
\end{definition}


Remembering that the usual Euclidean norm of the vector $y$ is defined by $\Vert y\Vert=\sqrt{y_1^2+y_2^2}$, we have that the function $F$ in the above definition can be interpreted as a generalization or perturbation of the usual Euclidean norm. That is, $F(x,y)$ can be thought of as the “norm” of the vector $y\in \mathbb{R}^2$ at the point $x\in \mathbb{B}^2$. In the language of Finsler geometry, this perturbed “norm” is called the Funk metric on the unit disk $\mathbb{B}^2$.

\begin{definition}[Finsler Arc Length]\label{def:arc}{\em
		Let $c:[a,b]\rightarrow U\subset\mathbb{R}^2$ be a piecewise regular curve. The \textit{(Finsler-type) arc length of $c$} is defined by
		\[ \mathscr{L}_F(c):=\int_a^bF(c(t),c'(t))dt,\]
		where $F$ is a Finsler metric. For any points $P, Q \in \mathbb{B}^2$, the \textit{distance} from $P$ to $Q$ induced by $F$ is defined as
		\[ d_F(P,Q):=\operatorname{inf}_c\mathscr{L}_F(c),\]
		where the infimum is taken over the set of all piecewise regular curves $c:[a,b]\to U$ such that $c(a)=P $ and $ c(b)=Q $.}
\end{definition}

The Finsler-type arc length, when $F$ is a Randers metric can be interpreted as traveling time of a boat sailing along the curve $c$, from $c(a)$ to $c(b),$ for example, the Funk metric models a boat sailing with unit speed in $\mathbb{B}^2$, where a wind current given by $W_x=(-x_1,-x_2)$ is present with speed less than 1. (see Section 3 in \cite{Chavez2021}). The Funk metric in $\mathbb{B}^2$ is a special case of Randers metrics. Some properties were studied in \cite{Shen2001} for more general cases. 

A Finsler metric $F$ on an open subset $U\subset \mathbb{R}^n$ is said to be {\it Projectively flat} if all geodesics are straight lines in $U.$

It is known that the Funk metric is projectively flat, it is, their \textit{shortest paths} are straight lines (see Example 9.2.1 in \cite{Shen2001}).  More details on geodesics and their relation to shortest paths can be found in Section 3.2 in \cite{ChernShen2005} and Section 2.3 in \cite{ChengShen2012}.\par 
The distance induced by the Funk metric is given by \eqref{eq1}, in Remark 4.2, \cite{Chavez2021} proved that $d_F$ given by \eqref{eq1} is not reversible ($d_F(P,Q)\neq d_F(Q,P)$) and it is not translation invariant, but the one is rotation invariant. Additionally, considering $O=(0,0)$, in Remark 5.3, \cite{Chavez2021} showed that \[\lim\limits_{\Vert Q\Vert \to 1} d_F(O,Q) = +\infty\;\mbox{ and }\;\lim\limits_{\Vert P\Vert \to 1} d_F(P,O) = \ln 2.\] In other words, if a boat starts from the origin of the disk towards the boundary, it will take an infinite time to reach the destination, meaning it never reaches the boundary. And if a boat starts from the boundary of $\mathbb{B}^2$ towards the origin of the disk, the minimum travel time is $\ln 2$ units of time.

Equation \eqref{eq1} can be more manageable using the following version of Theorem 5.1 in \cite{Chavez2021}.
\begin{theorem}[Theorem 2.3 in \cite{Chavez2024}]\label{maintheoremPF}{\em 
		Let $P, Q$ be points in $\mathbb{B}^2$ and $r\geq1$ be a real number, then
		\[d_F (P, Q) = \ln r \iff\left\Vert \dfrac{P}{r}-Q\right\Vert=\dfrac{r-1}{r},\]
		where $\Vert\cdot \Vert$ denotes the usual Euclidean norm.
	}
\end{theorem}
\section{Zermelo navigation problem}
It is worth noting that any Zermelo navigation problem in $ \Omega \subset \mathbb{R}^2 $ (including in broader domains like differentiable manifolds) results in a Randers metric on $ \Omega $ (or on larger domains). Furthermore, every Randers metric originates from a navigation problem (see Chapter 2 in \cite{ChengShen2012}).

In this section, we will explore Randers metrics derived from the Zermelo navigation problem modeled on an open subset $ \Omega $ of $ \mathbb{R}^2 $.

Suppose that a boat is pushed by an internal force (like a motor force) with the velocity vector $U_x$ of constant length, $||U_x||=1$.
Without an external velocity vector, the shortest paths are straight lines. In this case, the length of a straight line segment corresponds exactly to the travel time of the boat. Specifically, if $c:[0,t_0]\rightarrow \mathbb{R}^2$ is the position vector of the boat, such that $c'(t)=U_{c(t)}$ (a unit velocity vector), then:
\[\int_0^t \Vert c'(\tau)\Vert d\tau = \int_0^t 1d\tau = t.\]

Now, suppose that there exists an external velocity vector $ W_x $, like generated by the wind, with $\Vert W_x \Vert < 1$. This condition ensures that the boat can move in all directions. The combination of the two velocity vectors on the object, $ T_x = U_x + W_x $, gives the direction and speed of the object at the point $ x \in \Omega$. Once the internal velocity vector $ U_x $ with $\Vert U_x \Vert = 1$ is chosen, we have:
\begin{align}\label{eq2}
	\Vert T_x - W_x \Vert = \Vert U_x \Vert = 1.
\end{align}

For any vector $ y_x \in \mathbb{R}^2 $, there exists (see Figure \ref{fig:existF}) a unique solution $ F = F(x, y_x) > 0 $ to the following equation:
\begin{align}\label{eq3}
	\left\Vert \frac{y_x}{F(x, y_x)} - W_x \right\Vert = 1.
\end{align}

\begin{figure}[H]
	\centering 
	\includegraphics[width=6cm, trim= 10mm 35mm 10mm 35mm,clip]{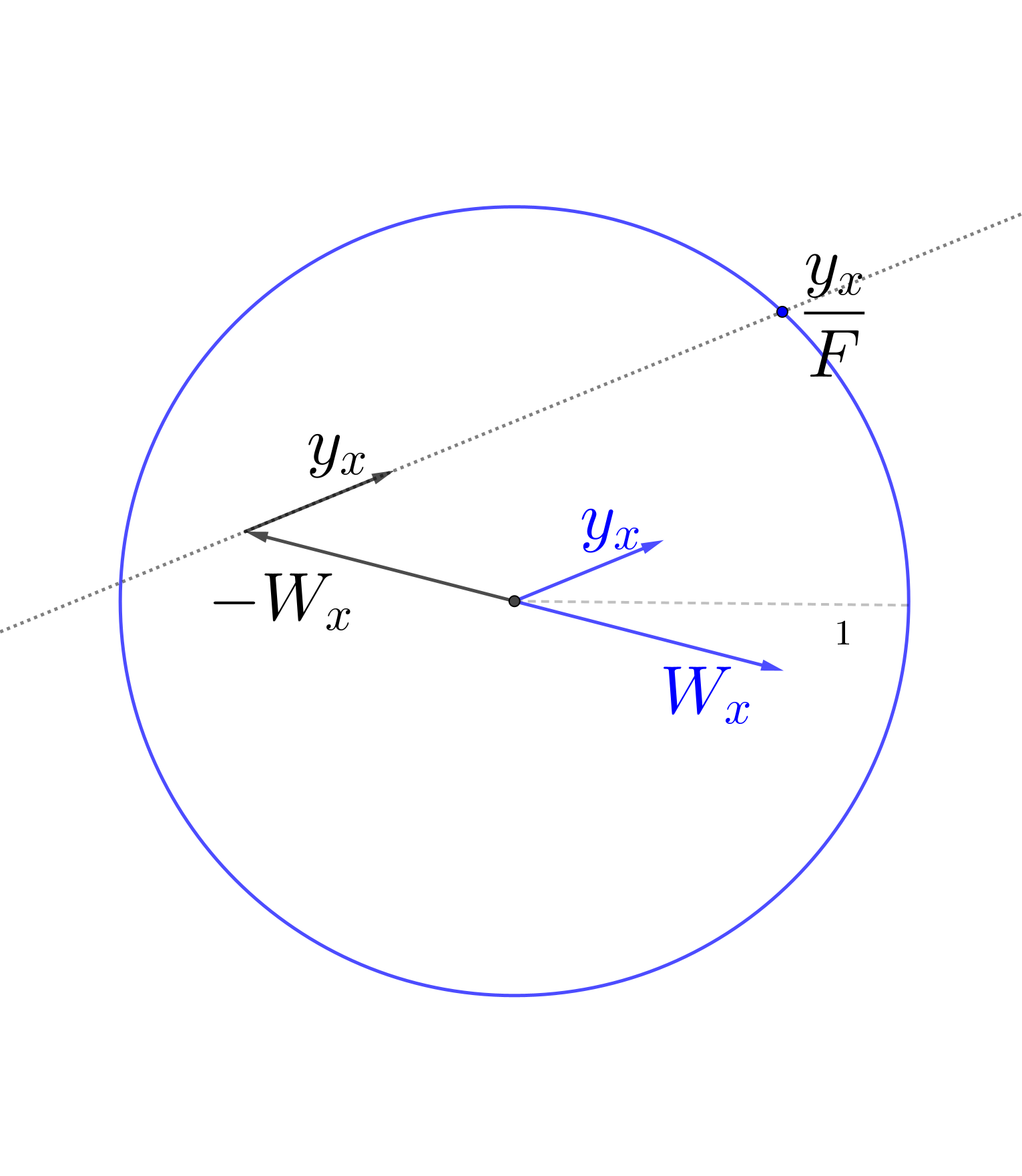} 
	\caption{Existence of $F$.}
	\label{fig:existF}
\end{figure}



By comparing \eqref{eq2} and \eqref{eq3}, we obtain:
\[ F(x, T_x) = 1. \]

If $ c:[0,t_0] \to \Omega $ is a smooth curve with $ c'(t) = T_{c(t)} $, then the arc length $\mathcal{L}_F(c)$ of $ c $ is equals to the travel time of the object along $ c $. Indeed:
\[ \mathcal{L}_F(c) = \int_0^{t_0} F(c(\tau), T_{c(\tau)}) d\tau = \int_0^{t_0} 1 d\tau = t_0. \]

Therefore, in the presence of an external force $ W_x $, the search for shortest paths is no longer in the Euclidean metric, but in the metric $ F $.

We will now derive an expression for the function $ F = F(x, y_x) $. To simplify notation, we will omit the subscript $ x $ from $ y_x $ and $ W_x $ going forward. Squaring both sides of Equation \eqref{eq3} and expanding the squared norm, we get:
\[ \frac{\Vert y \Vert^2}{F^2} - 2 \frac{\langle y, W \rangle}{F} + \Vert W \Vert^2 = 1. \]

Multiplying both sides of the above equality by $ F^2 $, we have:
\[ (1 - \Vert W \Vert^2) F^2 + 2 \langle y, W \rangle F - \Vert y \Vert^2 = 0. \]

This last equality is a quadratic equation, whose roots are given by:
\[ F = -\frac{\langle W, y \rangle}{1 - \Vert W \Vert^2} \pm \frac{\sqrt{\langle W, y \rangle^2 + \Vert y \Vert^2 (1 - \Vert W \Vert^2)}}{1 - \Vert W \Vert^2}. \]

Now, since:
\[ \sqrt{\langle W, y \rangle^2 + \Vert y \Vert^2 (1 - \Vert W \Vert^2)} \geq |\langle W, y \rangle| \geq \langle W, y \rangle. \]

Equality holds if and only if $ y = 0 $. Thus, there will always be one positive root and one non-positive root, ensuring that $ F > 0 $ for all $ y \neq 0 $. Thus, we obtain:
\begin{align}\label{eq4}
	F = \frac{\sqrt{\langle W, y \rangle^2 + \Vert y \Vert^2 (1 - \Vert W \Vert^2)}}{1 - \Vert W \Vert^2} - \frac{\langle W, y \rangle}{1 - \Vert W \Vert^2}.
\end{align}

Given a non-negative constant $\lambda$, we are going to define the $\lambda-$Funk metric considering $W=W_x=-\lambda(x_1,x_2)$ (see Figure \ref{fig:windcurrent}).

\begin{figure}[H]
	\centering 
	\includegraphics[width=7cm, trim= 100mm 170mm 100mm 180mm,clip]{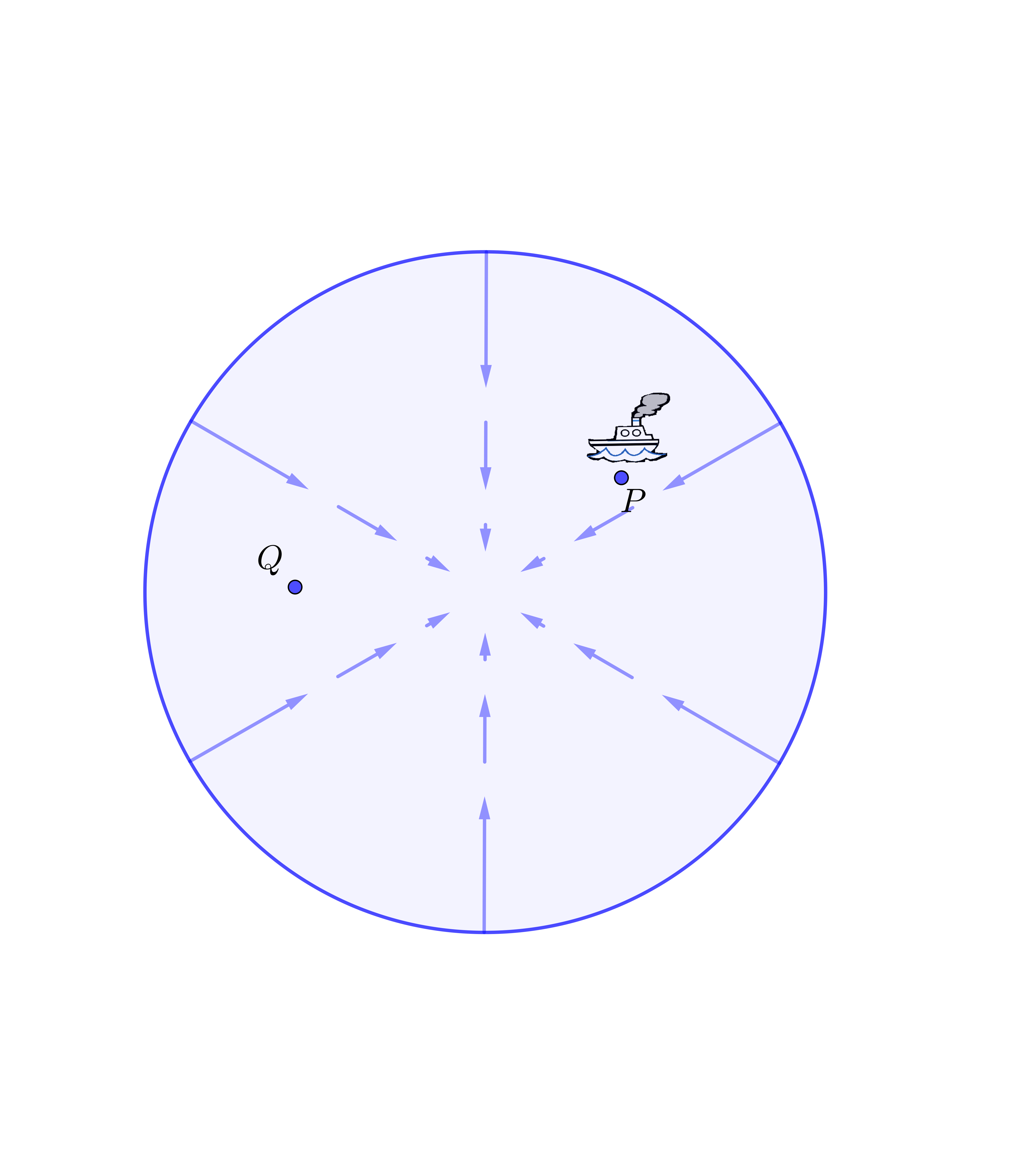} 
	\caption{Wind Current: $W_x=-\lambda x$ .}
	\label{fig:windcurrent}
\end{figure}

\begin{definition}\label{def:lambdafunk}
	The $\lambda-$Funk metric $F:\Omega_\lambda\times \mathbb{R}^n\to \mathbb{R}$ is defined by
	\begin{align}\label{eq5}
		F = \frac{\sqrt{\lambda^2\langle x, y \rangle^2 + \Vert y \Vert^2 (1 - \lambda^2\Vert x \Vert^2)}}{1 - \lambda^2\Vert x \Vert^2} + \frac{\lambda\langle x, y \rangle}{1 - \lambda^2\Vert x \Vert^2},
	\end{align}
	where   $\Omega_{\lambda} = \{ x \in \mathbb{R}^2; \| x \| < 1/{\lambda} \}$ for $\lambda> 0$, and $\Omega_0=\mathbb{R}^n$.
\end{definition}

\begin{remark}
	It is straightforward to prove that the function \eqref{eq5} satisfies all the properties of a Finsler metric in Definition \ref{def:Finsler}. And, note that $F$ in \eqref{eq5}, when $\lambda=0$ or $\lambda=1$ remain the classical Euclidean norm and the Funk metric in \eqref{FunkMetric} respectively.
\end{remark}

We recall the following result, which characterizes spherically symmetric Finsler metrics that are projectively flat.
\begin{theorem}[Theorem 5.1 in \cite{GuoMo2018}]\label{thmpflat}{\em 
		Let $ F = \Vert y\Vert\phi\left(\Vert x\Vert, \frac{\langle x, y\rangle}{\Vert y\Vert} \right) $ be a spherically symmetric Finsler metric in $ \mathbb{B}^n(\mu)=\{x\in\mathbb{R}^n;\;\Vert x\Vert<\mu\}$. Then, $ F $ is projectively flat if, and only if, $ \phi $ satisfies $ r \phi_{ss} - \phi_r + s\phi_{rs} = 0 $.}
\end{theorem}

\begin{theorem}\label{theo:flat}
	The $\lambda-$Funk metric defined in \eqref{eq5} is projectively flat.
\end{theorem}

\begin{proof}
	Note that  \eqref{eq5} can be rewritten as $F = \Vert y \Vert \phi\left(\Vert x \Vert, \frac{\langle x, y \rangle}{\Vert y \Vert}\right)$ where
	\begin{equation}\label{eq6}
		\phi(r,s) = \dfrac{\sqrt{1+\lambda^2(s^2-r^2)}}{1-\lambda^2r^2} + \dfrac{\lambda s}{1-\lambda^2r^2}.    
	\end{equation}
	By differentiating $\phi$ with respect to $r$ and $s$ in \eqref{eq6}, we obtain
	\begin{equation}\label{eq7}
		\phi_r = \dfrac{\lambda^2 r (1 + \lambda^2 [2s^2 - r^2])}{(1 - \lambda^2 r^2)^2 (1 + \lambda^2 [s^2 - r^2])^{1/2}} + \dfrac{2\lambda^3 sr}{(1 - \lambda^2 r^2)^2}
	\end{equation}
	and
	\begin{equation}\label{eq8}
		\phi_s = \dfrac{\lambda^2 s}{(1 - \lambda^2 r^2)(1 + \lambda^2 [s^2 - r^2])^{1/2}} + \dfrac{\lambda}{1 - \lambda^2 r^2}.    
	\end{equation}
	Now, by differentiating \eqref{eq7} and \eqref{eq8} with respect to $s$, we have
	\begin{equation}\label{eq9}
		\phi_{rs} = \dfrac{\lambda^4 rs (3 + \lambda^2 [2s^2 - 3r^2])}{(1 - \lambda^2 r^2)^2 (1 + \lambda^2 [s^2 - r^2])^{3/2}} + \dfrac{2\lambda^3 r}{(1 - \lambda^2 r^2)^2}    
	\end{equation}
	and
	\begin{equation}\label{eq10}
		\phi_{ss} = \dfrac{\lambda^2}{(1 + \lambda^2 [s^2 - r^2])^{3/2}}.  
	\end{equation}
	Thus, using \eqref{eq7}, \eqref{eq9}, and \eqref{eq10}, we obtain
	\begin{align*}
		r\phi_{ss}-\phi_r+s\phi_{rs}&=\dfrac{\lambda^2r}{(1 + \lambda^2 [s^2 - r^2])^{3/2}}-\dfrac{\lambda^2 r (1 + \lambda^2 [2s^2 - r^2])}{(1 - \lambda^2 r^2)^2 (1 + \lambda^2 [s^2 - r^2])^{1/2}} - \dfrac{2\lambda^3 sr}{(1 - \lambda^2 r^2)^2}\\ &\;\;\;\;+\dfrac{\lambda^4 rs^2 (3 + \lambda^2 [2s^2 - 3r^2])}{(1 - \lambda^2 r^2)^2 (1 + \lambda^2 [s^2 - r^2])^{3/2}} + \dfrac{2\lambda^3 rs}{(1 - \lambda^2 r^2)^2}\\&=\dfrac{\lambda^2r\left\{(1-\lambda^2r^2)^2-(1+\lambda^2[2s^2-r^2])(1+\lambda^2[s^2-r^2])+\lambda^2 s^2 (3 + \lambda^2 [2s^2 - 3r^2])\right\}}{(1 - \lambda^2 r^2)^2 (1 + \lambda^2 [s^2 - r^2])^{3/2}}\\&=\dfrac{\lambda^2r\left\{1+\lambda^2[3s^2-2r^2]+\lambda^4[2s^4-3r^2s^2+r^4]-(1+\lambda^2[2s^2-r^2])(1+\lambda^2[s^2-r^2])\right\}}{(1 - \lambda^2 r^2)^2 (1 + \lambda^2 [s^2 - r^2])^{3/2}}\\&=0.
	\end{align*}
	Therefore, by Theorem \ref{thmpflat}, we have that $\lambda-$Funk metric given by \eqref{eq5} is projectively flat on $\Omega_\lambda$.
\end{proof}

\section{Induced Distance}
Inspired by \cite{Chavez2021} we obtain the distance from $P$ to $Q$ denoted by $d_F(P,Q)$. 

\begin{theorem}\label{theo:dPQ}
	Let $\lambda>0$, $\Omega_\lambda=\mathbb{B}^2(1/\lambda)$. Given $P,Q\in \Omega_\lambda$, then, the distance (or traveling time) induced by the $\lambda-$Funk metric, for $P\neq Q$, is given by:
	\begin{align}\label{eq29}
		d_F(P,Q)&=\dfrac{1}{\lambda}\ln\left(\frac{\sqrt{\lambda^2\langle P, Q-P \rangle^2+(1-\lambda^2\Vert P\Vert^2)\Vert Q-P\Vert^2}-\lambda\langle P, Q-P \rangle}{\sqrt{\lambda^2\langle P, Q-P\rangle^2+(1-\lambda^2\Vert P\Vert^2)\Vert Q-P\Vert^2}-\lambda\langle Q, Q-P \rangle}
		\right),
	\end{align}
	where $\langle\cdot,\cdot\rangle$ and $\Vert \cdot \Vert$ are the usual inner product and the usual Euclidean norm, respectively, and $d_F(P,P)=0$.
\end{theorem}
\begin{proof}
	Let $P, Q \in \Omega_\lambda \subset \mathbb{R}^2$ be distinct points. By Theorem \ref{theo:flat} we consider the parametric curve $c:[0,1]\rightarrow \Omega_\lambda$ defined by:
	\begin{equation}\label{eq11}
		\begin{array}{rl}
			c(t) &= tQ + (1-t)P
		\end{array}
	\end{equation}
	which connects the points $P$ and $Q$. Note that $c(0)=P$ and $c(1)=Q$. Differentiating \eqref{eq11} with respect to $t$:
	\begin{equation}\label{eq12}
		c^\prime(t) = Q - P
	\end{equation}
	By the Definition \ref{def:arc}, the ``$\lambda-$Funk'' distance from $P$ to $Q$, denoted by $d_F(P,Q)$, is given by:
	\begin{equation}\label{eq13}
		d_F(P,Q)= \mathscr{L}_F(c)= \displaystyle\int_0^1 F(c(t), c'(t)) \, dt
	\end{equation}
	where $F$ is given by \eqref{eq5}. Substituting \eqref{eq11} and \eqref{eq12} into \eqref{eq13}, we have:
	\begin{align*}
		d_F(P,Q) = d_F(P,Q)_1 + d_F(P,Q)_2
	\end{align*}
	where
	\begin{equation}\label{eq14}
		d_F(P,Q)_1 = \int_0^1 \frac{\sqrt{\Vert Q-P \Vert^2 (1 - \lambda^2\Vert P \Vert^2) +\lambda^2 \langle P, Q-P \rangle^2}}{(1 - \lambda^2\Vert P \Vert^2) - 2\lambda^2 \langle P, Q-P \rangle t - \lambda^2\Vert Q-P \Vert^2 t^2} \, dt
	\end{equation}
	and
	\begin{equation}\label{eq15}
		d_F(P,Q)_2 = \int_0^1 \frac{\lambda\Vert Q-P \Vert^2 t + \lambda\langle P, Q-P \rangle}{(1 - \lambda^2\Vert P \Vert^2) - 2\lambda^2 \langle P, Q-P \rangle t - \lambda^2\Vert Q-P \Vert^2 t^2} \, dt.
	\end{equation}
	We define $k$ as follows to simplify the notation:
	\begin{align}\label{eq16}
		k = \Vert Q-P \Vert^2 (1 - \lambda^2\Vert P \Vert^2) + \lambda^2\langle P, Q-P \rangle^2.
	\end{align}
	Note:
	\[
	(1 - \lambda^2\Vert P \Vert^2) - 2 \lambda^2\langle P, Q-P \rangle t - \lambda^2\Vert Q-P \rVert^2 t^2 = -\lambda^2\Vert Q-P \Vert^2 \left[t^2 + 2 \frac{\langle P, Q-P \rangle}{\Vert Q-P \Vert^2} t - \frac{1 - \lambda^2\Vert P \Vert^2}{\lambda^2\Vert Q-P \Vert^2} \right].
	\]
	Now, factorizing the right-hand side of the above equation, we obtain:
	\begin{equation}\label{eq17}
		(1 - \lambda^2\Vert P \Vert^2) - 2 \lambda^2 \langle P, Q-P \rangle t -  \lambda^2\Vert Q-P \rVert^2 t^2 = -\lambda^2\Vert Q-P \Vert^2 (t - t_1)(t - t_2)
	\end{equation}
	where
	\begin{equation}\label{eq18}
		t_1 = \frac{- \lambda\langle P, Q-P \rangle-\sqrt{k}}{\lambda\Vert Q-P \rVert^2}
	\end{equation}
	and
	\begin{equation}\label{eq19}
		t_2 = \frac{- \lambda\langle P, Q-P \rangle+\sqrt{k}}{\lambda\Vert Q-P \rVert^2}.
	\end{equation}
	Substituting \eqref{eq16} and \eqref{eq17} into \eqref{eq14}, we have:
	\begin{align*}
		d_F(P,Q)_1 = -\dfrac{\sqrt{k}}{\lambda^2\Vert Q-P \rVert^2} \int_0^1 \frac{1}{(t - t_1)(t - t_2)} \, dt.
	\end{align*}
	This time, using the method of partial fractions, we have:
	\begin{align*}
		d_F(P,Q)_1 = -\dfrac{\sqrt{k}}{\lambda^2\Vert Q-P \rVert^2} \int_0^1 \left[ \frac{1}{(t_1 - t_2)(t - t_1)} - \frac{1}{(t_1 - t_2)(t - t_2)} \right] \, dt.
	\end{align*}
	Therefore, integrating, we obtain: 
	\begin{equation}\label{eq20}
		\begin{array}{rl}
			d(P,Q)_1 &= -\dfrac{\sqrt{k}}{\lambda^2\Vert Q-P \rVert^2} \left[ \dfrac{\ln |t - t_1|}{t_1 - t_2} - \dfrac{\ln |t - t_2|}{t_1 - t_2} \right]_0^1 \\ \\
			&= -\dfrac{\sqrt{k}}{\lambda^2\Vert Q-P \rVert^2} \dfrac{1}{(t_1 - t_2)} \left[ \ln \left| \dfrac{t - t_1}{t - t_2} \right| \right]_0^1 \\ \\
			&= -\dfrac{\sqrt{k}}{\lambda^2\Vert Q-P \rVert^2} \dfrac{1}{(t_1 - t_2)} \ln \left| \dfrac{t_2 (1 - t_1)}{t_1 (1 - t_2)} \right|.
		\end{array}
	\end{equation}
	Note that, by \eqref{eq18} and \eqref{eq19}, we have: 
	\begin{equation}\label{eq21}
		\dfrac{1}{t_1 - t_2} = -\dfrac{\lambda\Vert Q-P \rVert^2}{2 \sqrt{k}}.
	\end{equation} 
	Thus, substituting \eqref{eq21} into \eqref{eq20}, we have:
	\begin{equation}\label{eq22}
		d_F(P,Q)_1 = \dfrac{1}{2\lambda} \ln \left| \dfrac{t_2 (1 - t_1)}{t_1 (1 - t_2)} \right|.
	\end{equation}
	On the other hand, note 
	\begin{equation}\label{eq23}
		[(1 - \lambda^2\Vert P \rVert^2) - 2 \lambda^2\langle P, Q-P \rangle t - \lambda^2\Vert Q-P \Vert^2 t^2]' = -2 \lambda^2[\Vert Q-P \rVert^2 t + \langle P, Q-P \rangle].
	\end{equation}
	Thus, substituting \eqref{eq23} into \eqref{eq15}, we have:
	\begin{align*}
		d_F(P,Q)_2 = -\dfrac{1}{2\lambda} \int_0^1 \frac{[(1 -\lambda^2 \Vert P \rVert^2) - 2\lambda^2 \langle P, Q-P \rangle t - \lambda^2\Vert Q-P \rVert^2 t^2]'}{(1 - \lambda^2\Vert P \rVert^2) - 2\lambda^2 \langle P, Q-P \rangle t - \lambda^2\Vert Q-P \rVert^2 t^2} \, dt.
	\end{align*}
	By the Fundamental Theorem of Calculus, we obtain:
	\begin{align}\label{eq24}
		d_F(P,Q)_2 = -\frac{1}{2\lambda} \left[ \ln \left| (1 -\lambda^2 \Vert P \rVert^2) - 2\lambda^2 \langle P, Q-P \rangle t - \lambda^2\Vert Q-P \rVert^2 t^2 \right| \right]_0^1.
	\end{align}
	Thus, substituting \eqref{eq17} into \eqref{eq24}, we have:
	\[
	d_F(P,Q)_2=-\frac{1}{2\lambda}\left[\ln\left|-\lambda^2\Vert Q-P \Vert^2(t-t_1)(t-t_2)\right|\right]^1_0\]
	or, equivalently, we obtain:
	\begin{align}\label{eq25}
		d_F(P,Q)_2 = -\frac{1}{2\lambda} \ln \left| \frac{(1 - t_1)(1 - t_2)}{t_1 t_2} \right|.
	\end{align}
	Substituting \eqref{eq21} and \eqref{eq25} into \eqref{eq13}, we have:
	\begin{align}\label{eq26}
		\nonumber d_F(P,Q) &= \dfrac{1}{2\lambda} \ln \left| \dfrac{t_2 (1 - t_1)}{t_1 (1 - t_2)} \right| - \dfrac{1}{2\lambda} \ln \left| \dfrac{(1 - t_1)(1 - t_2)}{t_1 t_2} \right| \\
		&= \dfrac{1}{2\lambda} \ln \left( \dfrac{t_2^2}{(1 - t_2)^2} \right) \\\nonumber
		&=\dfrac{1}{\lambda} \ln \left| \dfrac{t_2}{t_2 - 1} \right|.
	\end{align}
	\begin{claim}\label{AF}{\em  $t_2>1$.}
	\end{claim}
	In fact, note that, from the definition of $t_2$ in \eqref{eq19}, we have $t_2 > 1$ if and only if 
	\begin{align}\label{eq27}
		\sqrt{\lambda^2\langle P, Q-P \rangle^2 + (1 - \lambda^2\Vert P \rVert^2) \Vert Q-P \rVert^2} > \lambda(\langle P, Q-P \rangle + \Vert Q-P \rVert^2).
	\end{align} 
	Now, two situations arise:
	\begin{itemize}
		\item If $\langle P, Q-P \rangle + \Vert Q-P \rVert^2 < 0$, then \eqref{eq27} is clearly true.
		\item If $\langle P, Q-P \rangle + \Vert Q-P \rVert^2 \geq 0$, then \eqref{eq27} is true if and only if  
		\begin{align*}
			\lambda^2\langle P, Q-P \rangle^2 + (1 - \lambda^2\Vert P \rVert^2) \Vert Q-P \rVert^2 &> \lambda^2\left(\langle P, Q-P \rangle^2 + 2 \langle P, Q-P \rangle \Vert Q-P \rVert^2 + \Vert Q-P \rVert^4\right) \\
			\Leftrightarrow (1 -\lambda^2 \Vert P \rVert^2) \Vert Q-P \rVert^2 &> 2 \lambda^2\langle P, Q-P \rangle \Vert Q-P \rVert^2 +\lambda^2 \Vert Q-P \rVert^4
		\end{align*}
		Since $P \neq Q$, dividing both sides of the above inequality by $\Vert Q-P \rVert^2$, we have:  
		\[
		\begin{array}{rcl}  
			1 - \lambda^2\Vert P \rVert^2 > 2\lambda^2 \langle P, Q-P \rangle +\lambda^2 \Vert Q-P \rVert^2 &\Leftrightarrow& 1 > \lambda^2\left(\Vert P \rVert^2 + 2 \langle P, Q-P \rangle + \Vert Q-P \rVert^2\right) \\ 
			&\Leftrightarrow& 1 > \lambda^2\Vert P + (Q-P) \rVert^2 \\ 
			&\Leftrightarrow& \dfrac{1}{\lambda^2} > \Vert Q \rVert^2,
		\end{array}
		\]
		which is also true.
	\end{itemize}
	Therefore, in either case, \eqref{eq27} is true. This concludes the proof of the claim.
	Using Claim \ref{AF} in \eqref{eq26}, we have that the distance from $P$ to $Q$ is given by:
	\begin{align}\label{eq28}
		d_F(P,Q) = \dfrac{1}{\lambda}\ln \left( \frac{t_2}{t_2 - 1} \right).
	\end{align}
	Replacing \eqref{eq19} in \eqref{eq28}, and using \[\langle P, Q-P\rangle+\Vert Q-P\Vert^2=\langle Q, Q-P\rangle\] we obtain the result.
\end{proof}
\begin{remark}{For $\lambda\neq 0$, we have \begin{enumerate}
			\item  $d_F$ is not symmetric. In fact, consider $O$ as the origin and $P$ any point in $\Omega_\lambda$ distinct from $O$. Thus, we observe:
			\[d_F(O,P)=-\dfrac{1}{\lambda}\ln(1-\lambda\Vert P\Vert)\neq \dfrac{1}{\lambda}\ln(1+\lambda\Vert P\Vert)=d_F(P,O).\] 
			Theorem \ref{maintheorem}, below, provides a better visualization of this asymmetry.
			\item $d_F$ is not invariant under translations. In fact, consider $T:\Omega_\lambda\to \Omega_\lambda$ given by $T(x)=x+P_0$ where $P_0\in \Omega_\lambda\setminus\{O\}$. Note:
			\[d_F(O,-P_0)=d_F(O,P_0)\neq d_F(P_0,O)=d_F(T(O),T(-P_0)).\]
			\item $d_F$ is invariant under rotations. It suffices to note that the Euclidean inner product and the Euclidean norm are invariant under rotations.
	\end{enumerate}}
\end{remark}

\section[geometryofdistance]{Geometry of the $\lambda-$Funk metric on $\Omega_\lambda$}
In this section, we will examine some geometric properties of basic geometry, such as the distance from one point to another, the distance from a point to a line, and the study of the circle, using the distance obtained in \eqref{eq29} induced by the $\lambda-$Funk metric \eqref{eq5}.
\begin{theorem}\label{maintheorem}{\em 
		Let $P,\;Q$ be points in $\Omega_\lambda$ ($\lambda> 0$) and $r\geq1$ a real number, then
		\begin{equation}\label{eq30}
			d_F (P, Q) =\dfrac{1}{\lambda} \ln r \iff\left\Vert \dfrac{P}{r}-Q\right\Vert=\dfrac{1}{\lambda}\dfrac{r-1}{r}.
		\end{equation}
	}
\end{theorem}
\begin{proof}
	If $r=1$, we have \[d_F(P,Q)=\dfrac{1}{\lambda}\ln 1=0\Leftrightarrow P=Q\Leftrightarrow \dfrac{1}{\lambda}\dfrac{1-1}{1}=0=\dfrac{1}{\lambda}\Vert P-Q\Vert=\dfrac{1}{\lambda}\left\Vert \dfrac{P}{1}-Q\right\Vert.\]
	If $r>1$, we have
	\[ \left\Vert \dfrac{P}{r}-Q\right\Vert=\dfrac{1}{\lambda}\dfrac{r-1}{r}\Leftrightarrow \left\Vert P\left(\dfrac{r-1}{r}\right)+(Q-P)\right\Vert=\dfrac{1}{\lambda}\dfrac{r-1}{r}\Leftrightarrow \left\Vert P+\left(\dfrac{r}{r-1}\right)(Q-P)\right\Vert^2=\dfrac{1}{\lambda^2}\]
	\[ \Leftrightarrow 1-\lambda^2\Vert P\Vert^2=\lambda^2\left(\dfrac{r}{r-1}\right)^2\Vert Q-P\Vert^2+2\lambda^2\left(\dfrac{r}{r-1}\right)\langle P,Q-P\rangle.\]
	Note that $P\neq Q$ since $r>1$, thus $\Vert P-Q\Vert\neq0$ and $r-1\neq0$. Therefore, multiplying the above equality by $(r-1)^2\Vert Q-P\Vert^2$ we get
	\[
	\Leftrightarrow (r-1)^2(1-\lambda^2\Vert P\Vert^2)\Vert Q-P\Vert^2=\lambda^2r^2\Vert Q-P\Vert^4+2\lambda^2r(r-1)\langle P,Q-P\rangle\Vert Q-P\Vert^2\]
	\begin{equation}\label{eq31}
		\Leftrightarrow \begin{array}{c}
			(r-1)^2\left(\lambda^2\langle P,Q-P\rangle^2+\Vert Q-P\Vert^2(1-\lambda^2\Vert P\Vert^2)\right)\\ =\lambda^2\left((r-1)\langle P,Q-P\rangle+r\Vert Q-P\Vert^2\right)^2.\end{array}
	\end{equation} 
	Now, since $\langle Q,rQ-P\rangle <\dfrac{1}{\lambda} \Vert rQ-P\Vert=\dfrac{r-1}{\lambda^2}$ (since $Q\in\Omega_\lambda$) and $r-1>0,$  we have
	\[\begin{array}{rl}\lambda^2\left[(r-1)\langle P,Q-P\rangle+r\Vert Q-P\Vert^2\right]&=\lambda^2\langle rQ-P,Q-P\rangle\\ &=\lambda^2\left[\Vert rQ-P\Vert^2-(r-1)\langle Q,rQ-P\rangle\right]\\&>\lambda^2[(r-1)^2-(r-1)^2]=0.\end{array}
	\]
	Thus, being $r-1>0$, \eqref{eq31} is equivalent to 
	\begin{equation}\label{eq32} \Leftrightarrow (r-1)\sqrt{\lambda^2\langle P,Q-P\rangle^2+\Vert Q-P\Vert^2(1-\lambda^2\Vert P\Vert^2)}=\lambda[(r-1)\langle P,Q-P\rangle+r\Vert Q-P\Vert^2].\end{equation}
	Since $Q\in \Omega_\lambda$ and $P\neq Q$ we have that 
	\[ \lambda^2[\langle Q,Q-P\rangle^2-\langle P,Q-P\rangle^2]=\lambda^2\left[\Vert Q\Vert^2-\Vert P\Vert^2\right]\Vert Q-P\Vert^2\neq(1-\lambda^2\Vert P\Vert ^2)\Vert Q-P\Vert^2\]
	\[\Leftrightarrow \sqrt{\lambda^2\langle P, Q-P\rangle^2+(1-\lambda^2\Vert P\Vert^2)\Vert Q-P\Vert^2}-\lambda\langle Q, Q-P \rangle\neq0.\]
	Therefore, we obtain that \eqref{eq32} is equivalent to
	\[\Leftrightarrow r=\frac{\sqrt{\lambda^2\langle P, Q-P \rangle^2+(1-\lambda^2\Vert P\Vert^2)\Vert Q-P\Vert^2}-\lambda\langle P, Q-P \rangle}{\sqrt{\lambda^2\langle P, Q-P\rangle^2+(1-\lambda^2\Vert P\Vert^2)\Vert Q-P\Vert^2}-\lambda\langle Q, Q-P \rangle}\]
	\[\Leftrightarrow\dfrac{1}{\lambda}\ln r=\dfrac{1}{\lambda}\ln\left(\frac{\sqrt{\lambda^2\langle P, Q-P \rangle^2+(1-\lambda^2\Vert P\Vert^2)\Vert Q-P\Vert^2}-\lambda\langle P, Q-P \rangle}{\sqrt{\lambda^2\langle P, Q-P\rangle^2+(1-\lambda^2\Vert P\Vert^2)\Vert Q-P\Vert^2}-\lambda\langle Q, Q-P \rangle}
	\right)=d_F(P,Q).\]
\end{proof}
Note that, in the proof of the previous theorem, the properties of the Euclidean norm were used. That is, the theorem remains valid for points $P$ and $Q$ in $\mathbb{B}^n(1/\lambda), \, n\geq 2$.
\begin{remark}\label{obsrota}{\em 
		Note that using equation \eqref{eq30} it is easier to show that $d_F$ is invariant under rotations, that is, 
		\begin{align}\label{eq33}
			d_F(RP,RQ)=d_F(P,Q),
		\end{align}	
		where $P,Q\in\Omega_\lambda$ and $ R= \left(\begin{array}{cc} \cos{\theta} & \sin{\theta} \\ -\sin{\theta} & \cos{\theta} \end{array}\right) $.}
\end{remark} 

\begin{remark}{\em
		Consider $ P=(0,0) $ in \eqref{eq30}, then \[\dfrac{1}{\lambda}\Vert Q\Vert=\dfrac{r-1}{r}.\] When $ \Vert Q\Vert\rightarrow \frac{1}{\lambda}$, that is, when $Q$ approaches the boundary of $\Omega_\lambda$ then $ r\rightarrow +\infty $, consequently  $d_F(P,Q) \rightarrow +\infty $. This means, together with Property 2 in Remark 2.8 in \cite{Chavez2021}, that from any point in  $\Omega_\lambda$, our boat will not be able to leave $ \Omega_\lambda$. On the other hand, consider $Q=(0,0)$ in \eqref{eq30}, then \[\Vert P\Vert=\dfrac{r-1}{\lambda}.\] When $ \Vert P\Vert\rightarrow \frac{1}{\lambda} $, then $ r\rightarrow 2 $, consequently  $d_F(P,Q) \rightarrow \frac{\ln 2}{\lambda}$. Thus, from the boundary of $\Omega_\lambda$, our boat can reach the origin in a time equal to $\dfrac{\ln 2}{\lambda}$.}
\end{remark}
\subsection{Circumference}
Since the Funk distance is not symmetric, that is, $d_F(P,Q)\neq d_F(Q,P)$, we have two interpretations for the notion of circumference, these are, the ``output'' from the center, called \textit{type 1}; and the ``input'' to the center, called \textit{type 2}. 
\begin{definition}{\em
		Given $P$ a point in $\Omega_\lambda$ and $r\geq1$ a real number, we define the \textit{type 1 Funk circumference}, with center $P$ and radius $\dfrac{\ln r}{\lambda}$, as the points $X\in\Omega_\lambda$ that satisfy the following equation: 
		\begin{align*}
			d_F(P,X)=&\dfrac{\ln r}{\lambda}.
	\end{align*} }
\end{definition}

By \eqref{eq30}, we have that the equation of the type 1 Funk circumference with center $P=(a,b)$ and radius $\dfrac{\ln r}{\lambda}$, is given by: 
\begin{align}\label{eq34} 
	\left(x_1-\dfrac{a}{r}\right)^2+\left(x_2-\dfrac{b}{r}\right)^2=\left(\frac{r-1}{\lambda r}\right)^{2}.
\end{align}
\begin{figure}[H]
	\centering 
	\includegraphics[width=10cm, trim= 100mm 200mm 100mm 200mm,clip]{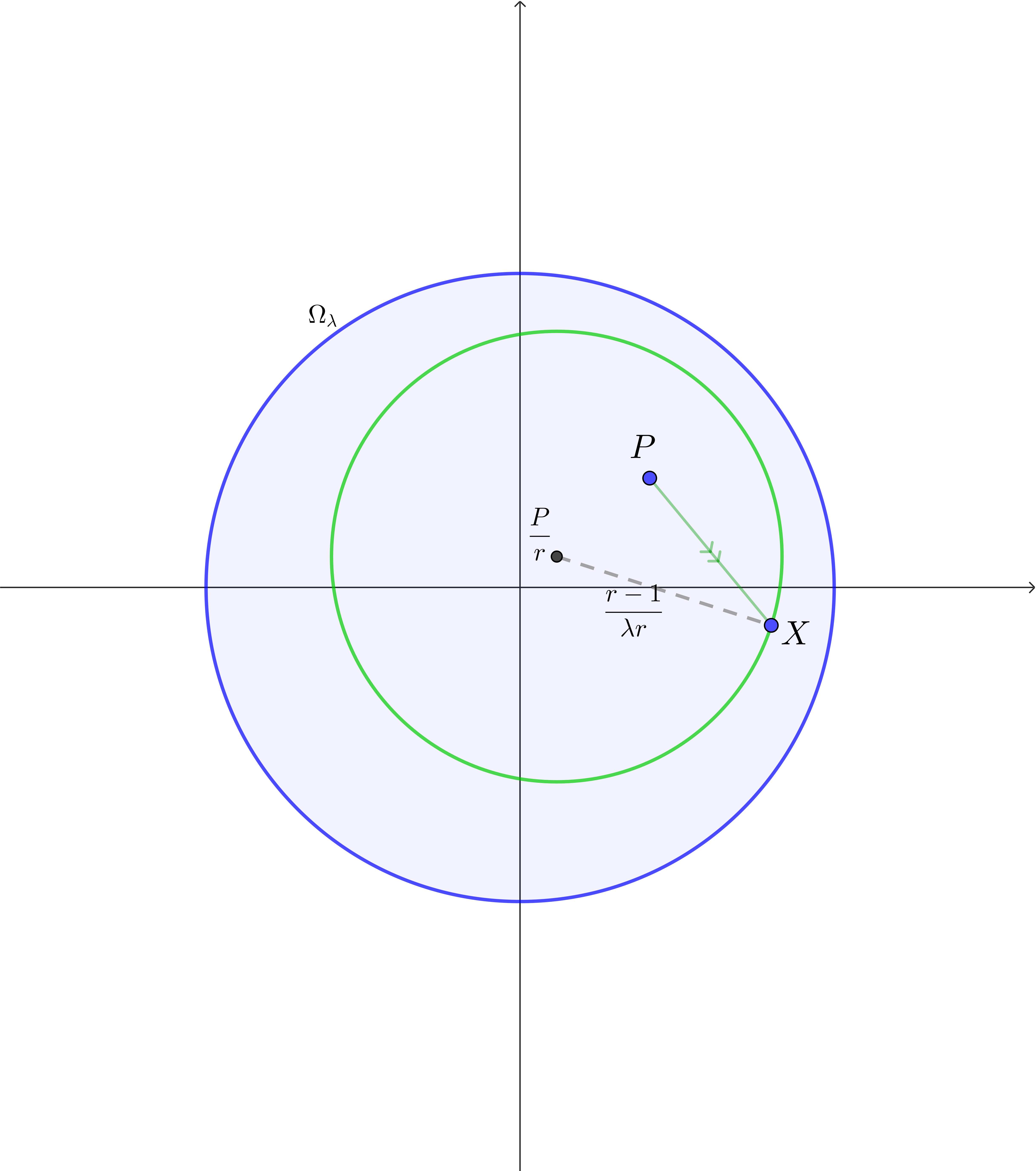} 
	\caption{$d_F(P,X)=\frac{\ln r}{\lambda}$ .}
	\label{fig:circ1}
\end{figure}

Equation \eqref{eq34} describes the graph of a Euclidean circle in $\mathbb{R}^2$ with center at $\frac{P}{r}$ and radius $\frac{r-1}{\lambda r}$ (see Figure \ref{fig:circ1}). 

\begin{definition}{\em
		Given $P$ a point in $\Omega_\lambda$ and $r\geq 1$ a real number, we define the \textit{type 2 Funk circumference}, with center $P$ and radius $\dfrac{\ln r}{\lambda}$, as the points $X\in\mathbb{B}^2$ that satisfy the following equation: 
		\begin{align*}
			d_F(X,P)=&\dfrac{\ln r}{\lambda}.
	\end{align*} }
\end{definition}
By \eqref{eq30}, we have that the equation of the type 2 Funk circumference with center $P=(a,b)$ and radius $\dfrac{\ln r}{\lambda}$, is part of the Euclidean circle with center at $(ra,rb)$ and radius $\dfrac{r-1}{\lambda}.$ (see Figure \ref{fig:circ2}).  
\begin{align*} 
	\left(a-\dfrac{x_1}{r}\right)^2+\left(b-\dfrac{x_2}{r}\right)^2=\left(\frac{r-1}{\lambda r}\right)^{2},
\end{align*}
or, equivalently,
\begin{align}\label{eq35}
	(x_1-ra)^2+(x_2-rb)^2=\left(\dfrac{r-1}{\lambda}\right)^2.
\end{align}

\begin{figure}[H]
	\centering 
	\includegraphics[width=10cm, trim= 100mm 200mm 100mm 200mm,clip]{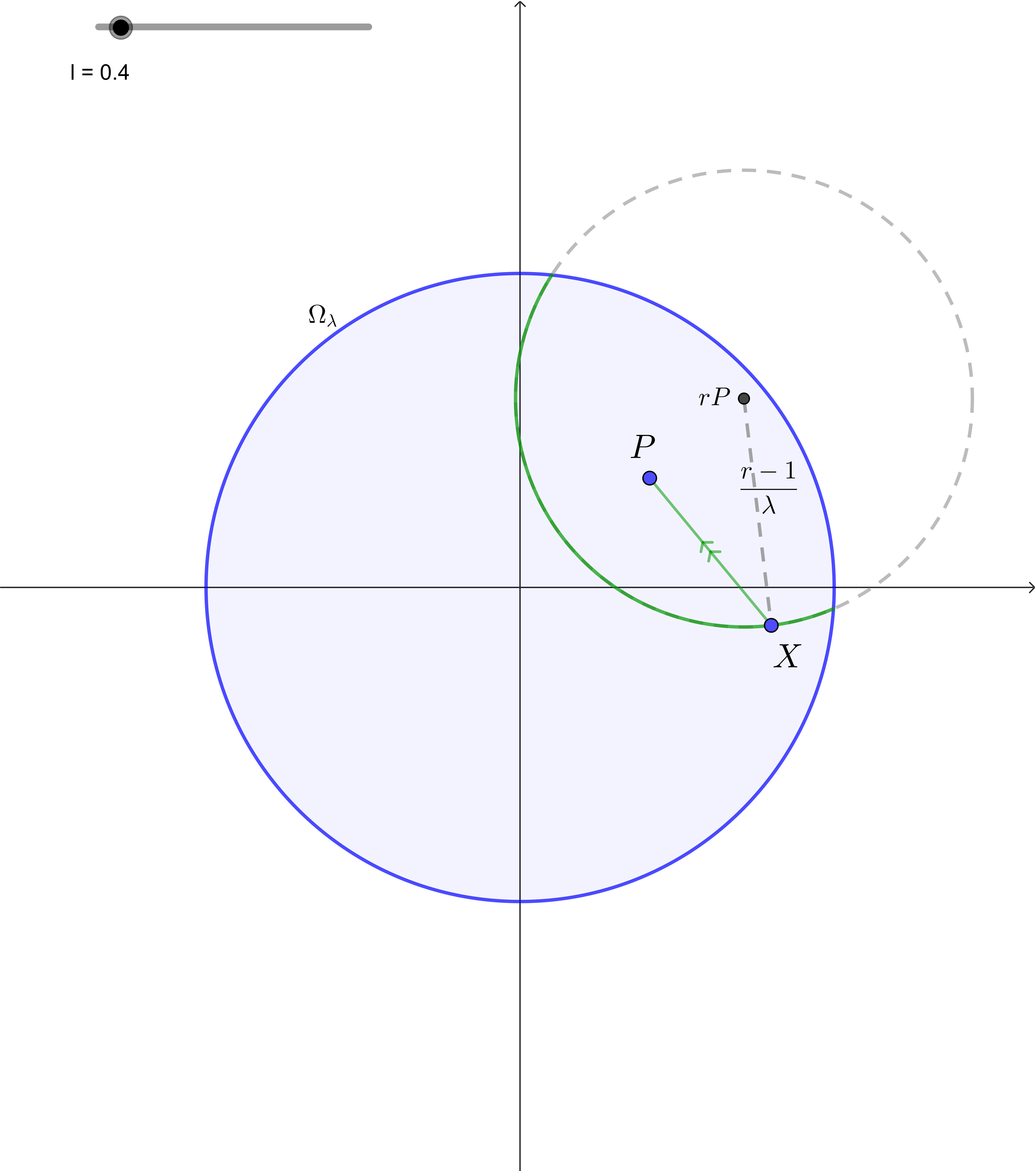} 
	\caption{$d_F(X,P)=\frac{\ln r}{\lambda}$ .}
	\label{fig:circ2}
\end{figure}

\subsection[distance]{$\lambda-$Funk distance from a Line to a Point}

\begin{definition}{\em
		We define the \textit{distance}, $d_F(s,Q),$ from the line $s$ to the point $Q$ as:
		\begin{align*}
			d_F(s,Q):=\min\{d(P',Q);\;P'\in s\}.
		\end{align*}
		We say that the point $P*\in s$ \textit{realizes the distance from $s$ to the point $Q$}, if \[d_F(P*,Q)\leq d_F(P',Q), \text{ for all } P'\in s.\]}
\end{definition}
\begin{theorem}\label{theo:dsP}
	The $\lambda-$Funk distance from the line $s:y=mx + c$ to the point $Q=(a,b)$ is given by:
	\begin{align}\label{eq39}
		d_{F}(s,Q)=\dfrac{\ln{r}}{\lambda}=\dfrac{1}{\lambda}\ln{\left(\frac{1-\lambda^2c\cos\theta(b\cos{\theta} - a \sin\theta) + \lambda|c\cos\theta - b\cos{\theta} + a \sin\theta|}{1- \lambda^2(b\cos{\theta} - a \sin\theta)^{2}}\right)},
	\end{align}
	where $\theta=\arctan m.$ And, this distance is realized by \begin{align}\label{eq40}
		P*=\left(\begin{array}{cc} \cos{\theta} & -\sin\theta \\ \sin\theta & \cos{\theta} \end{array}\right)\left(\begin{array}{c} {{a}{r}\cos{\theta} + {b}{r} \sin\theta} \\ c\cos{\theta} \end{array}\right).    
	\end{align}
\end{theorem}

\begin{proof}
	Given a line $s:\;x_2=mx_1+c$ and a point $Q=(a,b)\in\Omega_\lambda$, we will first analyze the particular case where $m=0$ and then analyze the general case. To determine the distance from the line $s$ to the point $Q$, we take a Funk circumference containing $Q$ and centered at some point on the line. The radius is $\dfrac{\ln{r}}{\lambda}$ and the center is at $P=(x_1,c)\in s$. According to Equation \eqref{eq35}, we have:
	\begin{align}\label{eq36}   
		\left(x_1- {r}a\right)^{2} + \left(c- rb\right)^{2}&=\left(\dfrac{r-1}{\lambda}\right)^2.
	\end{align}
	This is a quadratic equation in the variable $x_1$, and depending on $r$, it is possible to find two, one, or no solutions for $x_1$. We are only interested in finding a single solution, the one that minimizes the distance. Thus, from Equation \eqref{eq36}, we have the uniqueness conditions, $x_1=ra$ and
	\begin{align*}
		\left(c- rb\right)^{2}&=\left(\dfrac{r-1}{\lambda}\right)^{2}.
	\end{align*}
	Expanding the squares, we get a quadratic equation in $r$:
	\[(1-\lambda^2b^{2})r^{2}+2(\lambda^2bc-1)r+(1-\lambda^2c^{2})=0, \]
	whose roots are given by:
	\begin{align}\label{eq37}
		r=\frac{(1-\lambda^2bc)\pm \sqrt{(1-\lambda^2bc)^{2}-(1-\lambda^2b^{2})(1-\lambda^2c^{2})}}{1-\lambda^2b^{2}}. 
	\end{align} 
	Note that:
	\[
	(1-\lambda^2bc)^{2}-(1-\lambda^2b^{2})(1-\lambda^2c^{2})=\lambda^2(c-b)^{2}.
	\]
	Thus, equation \eqref{eq37} becomes $r=\dfrac{1-\lambda^2bc\pm \lambda|c-b|}{1-\lambda^2b^{2}}.$
	Remembering that $r>1$, we obtain:
	\begin{align}\label{eq38}
		r=\frac{1-\lambda^2bc+\lambda|c-b|}{1-\lambda^2b^{2}}.
	\end{align}
	Thus, the Funk distance from the line $s$ to the point $Q=(a,b)$, for the case where $m=0$, is given by:
	\begin{align*}
		d_F(s,Q)=\dfrac{\ln{r}}{\lambda}=\dfrac{1}{\lambda}\ln\left(\dfrac{1-\lambda^2bc+\lambda|c-b|}{1-\lambda^2b^2}\right).
	\end{align*}
	
	This distance is realized by the point, $P*=(ar,c)\in s.$
	We will later see how to determine this distance for any line $s$. To do this, we first observe that the force field is symmetric with respect to the origin (see equation \eqref{eq33}). Thus, we can apply a rotation to the axes and find values for $x_1$ and $r$ based on the results already obtained.
	
	Given a line $s:\;x_2=mx_1+c$, we rotate the coordinate axes by $\theta$ degrees, such that $m=\tan\theta$. Note that, in the new rotated axes, the line $s$ is a horizontal line (parallel to the $x_1$ axis), falling into the particular case where $m=0$. We determine the new coordinates of $Q=(a,b)$:
	\begin{align*}
		\left(\begin{array}{c} a' \\ b' \end{array}\right)= \left(\begin{array}{cc} \cos{\theta} & \sin{\theta} \\ -\sin{\theta} & \cos{\theta} \end{array}\right)\left(\begin{array}{c} a \\ b \end{array}\right)
	\end{align*}
	Thus, the new coordinates of $Q$ are $(a \cos \theta + b \sin \theta, b \cos\theta - a \sin \theta)$. 
	
	From \eqref{eq36} and \eqref{eq38}, we obtain:
	\[r=\frac{1-\lambda^2c\cos\theta(b\cos{\theta} - a \sin\theta) + \lambda|c\cos\theta - b\cos{\theta} + a \sin\theta|}{1-\lambda^2 (b\cos{\theta} - a \sin\theta)^{2}}\;\;\mbox{ and }\;\;x=r(a\cos{\theta} + b \sin\theta).\]
	
	Thus, we obtain the Funk distance from the line $s$ to the point $Q$, given by \eqref{eq39}. Moreover, this distance is realized by the point $P**=(r(a\cos{\theta} + b \sin\theta), c\cos{\theta}).$ And, to obtain the point on the original line (without rotation), it is sufficient to rotate back.
	
\end{proof}
\begin{example}
	Considering the vector field $W(x_1,x_2)=-2/5(x_1,x_2)$, the line $s: x+1$, and the point $Q=(1,1/10), $ (see Figure \ref{fig:distsP}). From Theorem \ref{theo:dsP}, we have, $d_F(s,Q)=1.36$, and $P^*=(0.45,1,45).$    \begin{figure}[H]
		\centering 
		\includegraphics[width=10cm, trim= 50mm 100mm 50mm 100mm,clip]{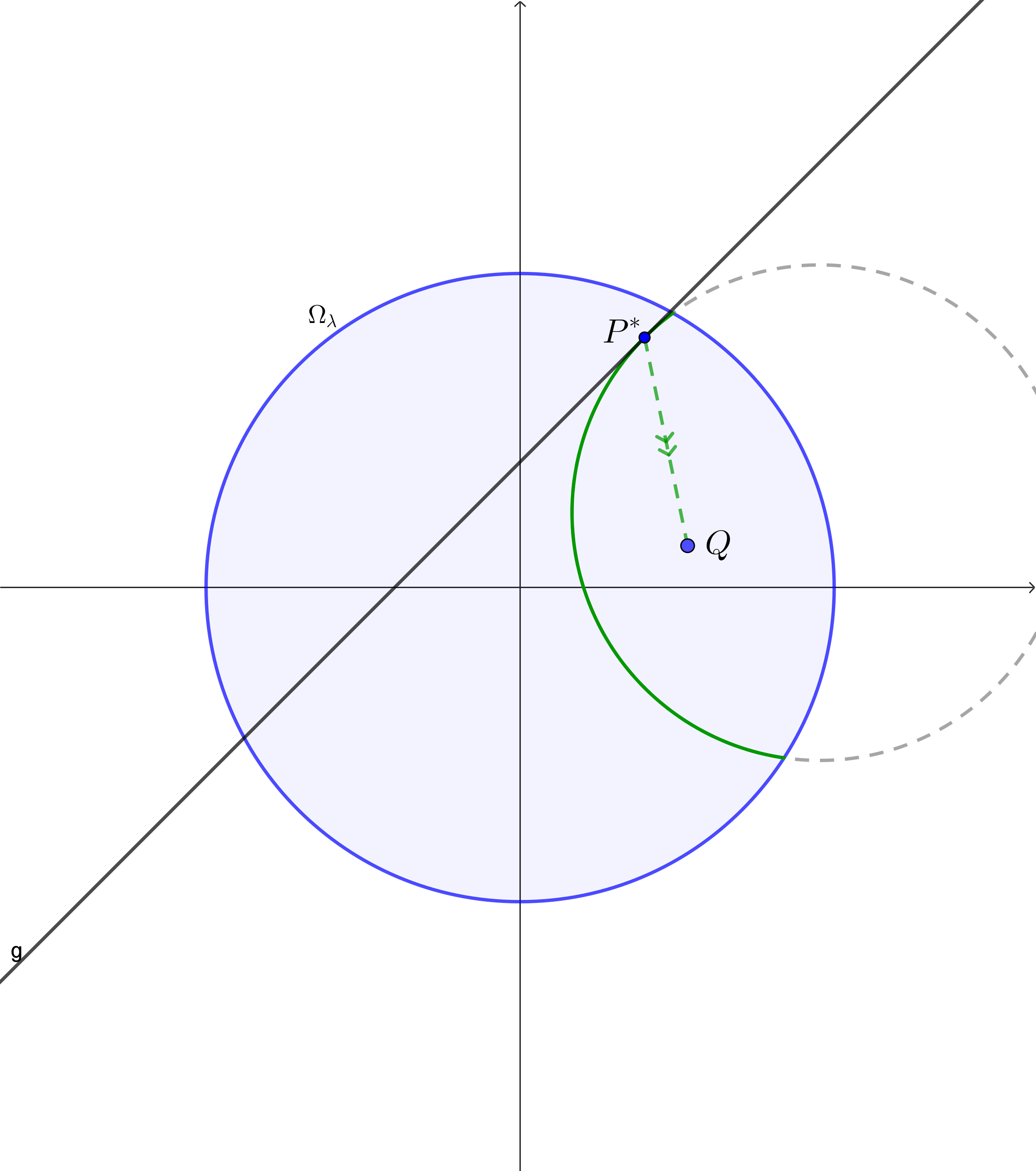} 
		\caption{$d_F(s,Q)=d_F(P^*,Q)= 1.36$ .}
		\label{fig:distsP}
	\end{figure}
\end{example}
\subsection[distancePtos]{$\lambda-$Funk Distance from a Point to a Line}

\begin{definition}{\em
		We define the \textit{distance}, $d_F(P,s),$ from the point $P$ to the line $s$ as:
		\begin{align*}
			d_F(P,s):=\min\{d(P,P');\;P'\in s\}.
		\end{align*}
		We say that the point $P*\in s$ \textit{realizes the distance from $P$ to the line $s$}, if \[d_F(P,P*)\leq d_F(P,P'), \text{ for all } P'\in s.\]}
\end{definition}
\begin{theorem}\label{theo:distPs}
	The $\lambda-$Funk distance from the point $P=(a,b)$ to the line $s:mx+c$ is given by:
	\begin{align}\label{eq42}
		d_F(P,s)=\dfrac{\ln{r}}{\lambda}=\dfrac{1}{\lambda}\ln\left(\frac{1-\lambda^2c\cos\theta(b\cos{\theta} - a \sin\theta) + \lambda|c\cos\theta - b\cos{\theta} + a \sin\theta|}{1-\lambda^2c^{2}\cos^2\theta}\right),
	\end{align}
	where $\theta = \arctan m.$ And, this distance is realized by
	\begin{align}\label{eq43}
		Q*=\left(\begin{array}{cc} \cos{\theta} & -\sin\theta \\ \sin\theta & \cos{\theta} \end{array}\right)\left(\begin{array}{c} {\frac{a}{r}\cos{\theta} + \frac{b}{r} \sin\theta} \\ c\cos{\theta} \end{array}\right).    
	\end{align}
\end{theorem}

\begin{proof}
	Given the point $P=(a,b)$ and the line $s:\;x_2=mx_1+c$, we will proceed similarly as we did before to determine the Funk distance from the line to the point. We take a Funk circumference centered at $P$, with radius $\dfrac{\ln{r}}{\lambda}$ and intersecting with the line $s$. We take the point $Q\in s$ to be one of these intersections. First, we will verify the particular case where $m=0$. Thus, by Equation \eqref{eq30}, we have:
	\begin{align}\label{eq41}
		\left(x_1-\dfrac{a}{r}\right)^{2}+\left(c-\dfrac{b}{r}\right)^{2}=\left(\dfrac{r-1}{\lambda r}\right)^{2}.
	\end{align}
	Similarly, this is a quadratic equation in the variable $x_1$, and depending on $r$. However, we are only interested in finding a single solution, the one that minimizes the distance. 
	To obtain a unique solution of \eqref{eq41}, the discriminant of the quadratic equation in $x_1$ above must be equal to zero. Consequently, we will have a quadratic equation in $r$:
	\begin{align*}
		r^2(\lambda^2c^2-1)+2r(1-bc\lambda^2)+\lambda^2b^2-1=0,
	\end{align*}
	whose roots are given by: 
	\begin{align*}
		r=\frac{(\lambda^2bc-1)\pm \sqrt{(\lambda^2bc-1)^{2}-(\lambda^2c^{2}-1)(\lambda^2b^{2}-1)}}{\lambda^2c^{2}-1}.
	\end{align*}
	We observe: \[
	(\lambda^2bc-1)^{2} - (\lambda^2c^{2}-1)(\lambda^2b^{2}-1) =\lambda^2(c-b)^2.\]
	
	Thus, $r=\dfrac{(\lambda^2bc-1)\pm \lambda|c-b|}{\lambda^2c^2-1}.$ To ensure that $r>1$, we have $r=\dfrac{\lambda^2bc-1- \lambda|c-b|}{\lambda^2c^{2}-1}$ and $x_1=\dfrac{a}{r}.$ Thus, the Funk distance from the point $P$ to the line $s$ (case where $m=0$) is given by:
	\begin{align*}
		d_{F}(P,s)=\dfrac{\ln{r}}{\lambda}=\dfrac{1}{\lambda}\ln{\left(\dfrac{\lambda^2bc-1-\lambda|c-b|}{\lambda^2c^2-1}\right)}.
	\end{align*}
	
	This distance is realized by the point $Q^{*}=\left(\dfrac{a}{r},c\right)$.
	
	For the case where $s$ is any line, rotate the axes, as we did before, and we obtain \eqref{theo:distPs} which is realized by the point $Q^{**}=\left(\dfrac{a\cos{\theta} + b \sin\theta}{r}, c\cos{\theta}\right)$. Then, we rotate back to obtain the point on the original line.
\end{proof}
\begin{example}
	Considering the vector field $W(x_1,x_2)=-2/5(x_1,x_2)$, the line $s: x+1$, and the point $Q=(1,1/10), $ (see Figure \ref{fig:distPs}). From Theorem \ref{theo:distPs}, we have, $d_F(P,s)=1.4$, and $Q^*=(-0.19,0.81).$    \begin{figure}[H]
		\centering 
		\includegraphics[width=10cm, trim= 50mm 100mm 50mm 100mm,clip]{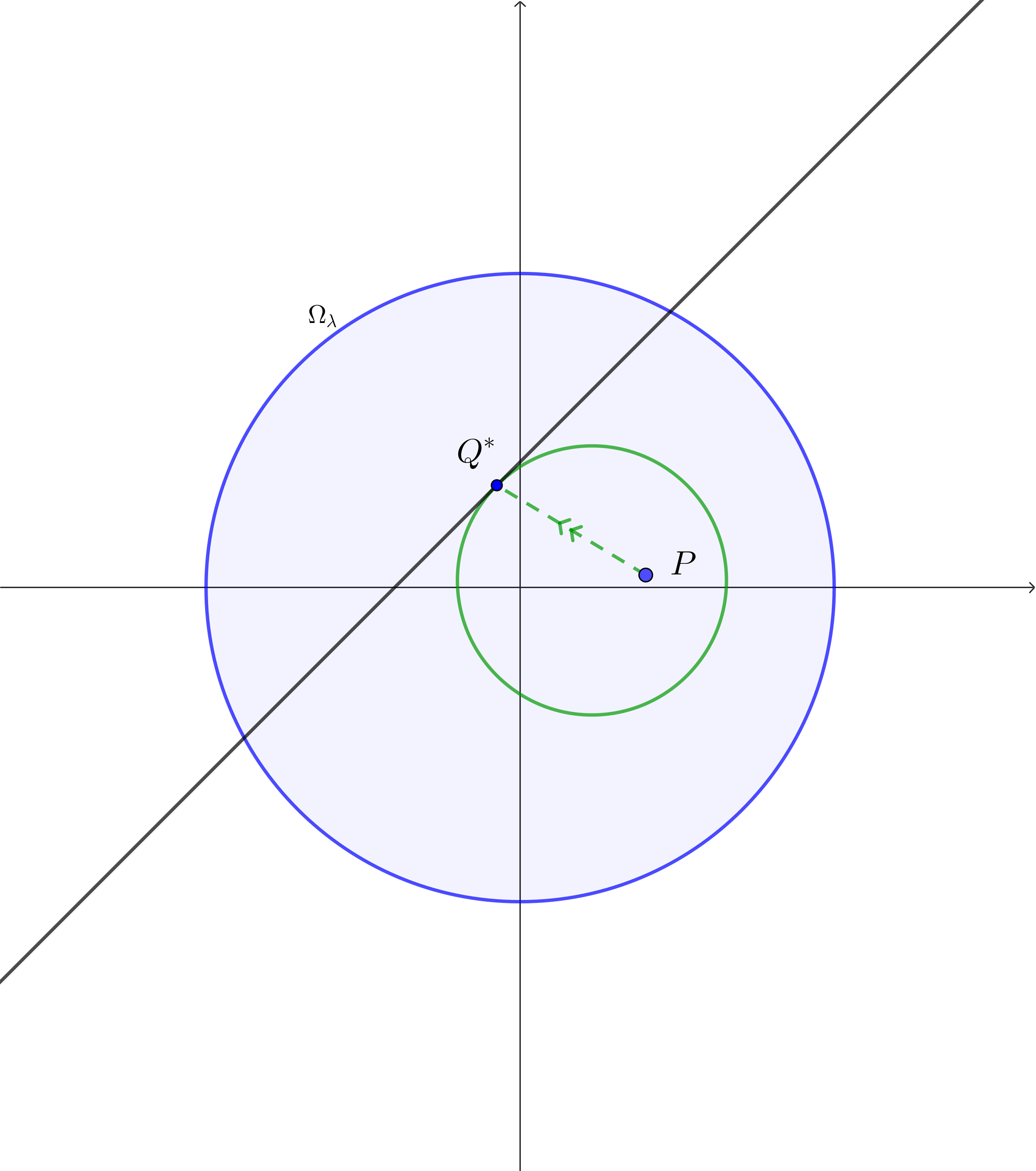} 
		\caption{$d_F(P,s)=d_F(P,Q^*)= 1.4$ .}
		\label{fig:distPs}
	\end{figure}
\end{example}

\section{Conclusion}
The Zermelo's navigation problem considered in this work was a boat navigating through a lake-like suitable disk, and the wind current was modeled by vector field $W_\lambda(x_1,x_2)=\lambda(-x_1,-x_2),$ where $\lambda \geq 0.$ With this, we defined the $\lambda-$Funk metric (Definition \ref{def:lambdafunk}) which generalize the usual Euclidean norm ($\lambda=0$) and the Funk metric given by \eqref{eq1} ($\lambda=1$). We prove that $\lambda-$Funk metric is spherically symmetric Finsler metric, with this, we prove the shortest path are straight lines (Theorem \ref{theo:flat}). The time traveling or the called $\lambda-$Funk distance induced by the $\lambda-$Funk metric was obtained in Theorem \ref{theo:dPQ}. Although the expression of this distance is complicated, Theorem \ref{maintheorem} gives us a useful tool to obtain the equation of the circumferences (equations \eqref{eq34} and \eqref{eq25}), the distance from line to point (Theorem \ref{theo:dsP}) and from point to line (Theorem \ref{theo:distPs}).

\end{document}